\documentclass[11pt]{amsart}
\usepackage[T1]{fontenc}
\usepackage{gfsartemisia-euler}
\usepackage{tikz}
\usetikzlibrary{calc}
\usepackage{graphicx}
\usepackage{setspace}
\usepackage{marginnote}
\usepackage[top=3.5cm, bottom=2.5cm, outer=2.5cm, inner=2.5cm]{geometry}
\usepackage{hyperref}
\hypersetup{linkcolor=red,
citecolor=black
}

\theoremstyle{plain}
\newtheorem{theorem}{Theorem}[section]
\newtheorem{theorem*}{Theorem}

\newtheorem{lemma}[theorem]{Lemma}
\newtheorem{proposition}[theorem]{Proposition}

\theoremstyle{definition}
\newtheorem{definition}[theorem]{Definition}

\newtheorem{remark}[theorem]{Remark}
\newtheorem{example}{Example}[section]

\newcommand{\rac}{{\mathsf{root}}}

\newcommand{\rt}{\emptyset}

\newcommand{\p}{\partial}
\newcommand{\EE}{{\mathbb{E}}}
\newcommand{\NN}{{\mathbb{N}}}
\newcommand{\ZZ}{{\mathbb{Z}}}
\newcommand{\PP}{{\mathbb{P}}}
\newcommand{\pr}[1]{\PP\left[\mspace{1mu} #1\mspace{1mu}\right]}
\newcommand{\ex}[1]{\EE\left[\mspace{1mu} #1\mspace{1mu}\right]}
\newcommand{\fword}[1]{{\emph{#1}}}

\numberwithin{equation}{section}

\title{Amenability of trees}
\author{B. Forghani}
\address[B. Forghani]{Department of Mathematics, University of Connecticut, Connecticut, USA. }
\email{behrang.forghani@uconn.edu}
\author{K. Mallahi-Karai}
\address[K. Mallahi-Karai]{Mathematics Department, Jacobs University of Bremen, Germany}
\email{kmallahika@jacobs-university.de}
\dedicatory{Dedicated to Wolfgang Woess on the occasion of his $60$th birthday}

\begin{document}
\maketitle
\begin{abstract}
We will give a criterion for the amenability of arbitrary locally finite trees. The criterion is based
on the trimming operator which is defined on the space of trees. As an application, we obtain a necessary and
sufficient condition for that amenability of Galton-Watson trees.

\end{abstract}
\section*{Introduction}
The notion of amenability for groups emerged out of von Neumann's effort \cite{Neumann} in 1929 to
find the underlying reason for Hausdorff's paradox. It is his observation that once the dimension of a Euclidean space $E$ exceeds two, the group of isometries of $E$ will contain a copy of the free group on two generators and hence fails to be amenable. This copy of the free group can then be used to carry out various paradoxical decompositions which are analogous to Hausdorff's original result.

Since its inception, the theory of amenable groups has been explored and the notion of amenability has been extended to a broad class of algebraic objects. Zimmer \cite{Zimmer78} found that a non-amenable group may have actions that share many properties of the actions of amenable groups and thus initiated the theory of amenable group actions.

Amenability of a discrete group can also be formulated in terms of its Cayley graph. Before giving this definition, let us recall the notion of the Cayley graph for a ( finitely generated) group. Recall that if $G$ is a group generated by a (symmetric) set $\Sigma$ of generators, then Cayley graph
$ \Gamma= \Gamma(G,\Sigma)$ is an undirected graph with the underlying set of $G$ as its vertex set in which vertices
$g_1, g_2 \in G$ form an edge when $g_1^{-1}g_2 \in \Sigma$. For instance, one can readily see that the Cayley graph of a free group on $k$ generators with respect to the standard generating set is isomorphic to the (unique) $2k$-regular tree. For a subset $A$ of vertices of $ \Gamma$, the boundary $ \partial A$, by definition, consists of those vertices in $A$  that have a neighbor in $V( \Gamma) \setminus A$. A family
$A_n \subseteq G$ is then called a F{\o}lner family, if $ | \partial A_n|/|A_n| \to 0$, as $n \to \infty$.
It is a classical result that a finitely generated group $G$ is amenable iff such a F{\o}lner family exists, see \cite{Folner55}. This definition lends itself to using many other methods to establish the amenability of a group. For instance, in various works by Bartholdi, Kaimanovich, Nekrashevych, Amir, Angel, and Virag (\cite{Kaimanovich2005},
\cite{BKN}, \cite{AAV}) random walks have been used to prove the amenability of several self-similar groups .
Many definitions and statements about amenable groups carry over almost verbatim to amenable graphs.
For instance, Gerl \cite{Gerl87} showed that a connected graph is amenable if and only
if the spectral radius of simple random walk on the graph is strictly less than 1. This result is a generalization of Kesten's result for amenable groups \cite{Kesten59}.

Note that the Cayley graph of a group $G$ is homogeneous, i.e., its automorphism group (containing $G$ as the subgroup of ``internal symmetries'') acts transitively on the vertex set of $G$. In particular, the only trees that can be realized as Cayley graphs are those of free products of cyclic group, which, except in trivial cases, are non-amenable.
In this note, we will take up the question of characterizing amenability for arbitrary locally finite trees. Our point of departure is a group of results proved by Gerl and Woess who investigated this property for trees that do not have degree-one vertices.
Recall that a branch is a vertex with degree at least three.
It is clear that a graph which contains arbitrarily long paths without branches is amenable.
Gerl \cite{Gerl86} proved a tree without any leaves with uniformly bounded degrees is amenable if and only if there are arbitrarily long paths without any branch as induced subgraphs. Later, Woess \cite[p.~114]{Woess2000} improved this result by dropping the uniform finiteness condition. (see Theorem~\ref{thesis}). One can easily see that the assumption that $T$ has not vertices of degree one cannot be dropped (see Example~\ref{amenable trees with leaves}).

Our first theorem extends this characterization to arbitrary trees and, \fword{en passant}, also supplies a rather elementary and ``probability-free'' proof of that result too. In order to state the theorem, we will need to introduce some new terminology. Our central new concept is the {\it trimming} operator
$\Theta$  defined on the space of countable trees. Intuitively, trimming a tree amounts to removing all vertices of degree $1$. Hence, the set of fixed points of this operator are precisely the trees without leaves (but see Example \ref{trimmed indefinitely}). Using the trimming operator, we will define \emph{inessential subtrees} of a trees which are, roughly speaking, finite subtrees hanging from a vertex of the main tree. Next to long paths, inessential subtrees can provide another source of F{\o}lner sets for infinite trees. Our theorem roughly says that these are the only underlying reasons for the amenability of a tree:

\begin{theorem*}\label{main}
Let $T$ be an infinite tree. Then $T$ is amenable if and only if $T$ contains arbitrarily large inessential trees or for some $k \ge 1$, $\Theta^k(T)$ has arbitrarily long paths. Moreover, the former is always the case if $T$ can be trimmed indefinitely.
\end{theorem*}

This theorem can then be applied in a probabilistic setting. There are many models for random trees. The oldest and perhaps the most well-known one is the family of Galton-Watson trees.
In this context, we have the following theorem:

\begin{theorem*}\label{galton}
The Galton-Watson tree $ {\mathcal T}$ associated to the probability distribution $(p_i)_{i \ge 0}$ is almost surely amenable, if and only if
$\pr{X_I \le 1}=p_0+p_1>0$.
\end{theorem*}

This paper is organized as follows. In Section \ref{prem}, we will define the graph-theoretical terminology that is freely used throughout the paper. Section \ref{amen} is devoted to stating general facts about amenability. A simple proof of Theorem \ref{main} in the special case of trees without leaves is also given in this section. In Section \ref{trim},
the trimming operator and inessential trees are introduced and studied. Finally, in Section \ref{gw}, we show an application of Theorem \ref{main} in the context of Galton-Watson trees.

\section*{acknowledgment} {Authors would like to thank V. Kaimanovich for his useful comments on the first draft of this article.
We would like to thank W. Woess for his comments about the history of amenability of trees and pointing out reference \cite{Woess2000}.

\section{Preliminaries}\label{prem}
In this section, we will define the basic graph-theoretic terminology and set the notations used in this paper. A (undirected) graph $G$ consists of a non-empty set $V(G)$ called the vertices and
a family of $2$-element subsets of $V(G)$, called the edges of $G$, and denoted by $E(G)$. For brevity, the edge $\{ u, v \}$ with vertices (also called endpoints) $u$ and $v$ will be denoted by $uv$, hence
$uv=vu$. The set of neighbors of a vertex $v$, denoted by $N(v)$, consists of the vertices $u \in V(G)$ with $uv \in E(G)$. Correspondingly, for a subset $A \subseteq V(G)$, we set $N(A)= \bigcup_{v \in A} N(a)$. The degree of a vertex is given by $\deg v= | N(v)|$, where $|X|$ denotes the cardinality of set $X$. All the graphs considered in this paper are assumed to be locally finite, that is, $\deg v < \infty$ for all $v \in V(G)$.
A leaf is a vertex of degree $1$. The set of leaves of a graph $G$ is denoted by $L(G)$.
For two leaves $u,v \in G$, we write $u \sim v$ if $N(u)=N(v)$. A branch is a vertex with degree strictly more than 2.
For a non-empty subset $A \subseteq V(G)$, the induced subgraph $G[A]$ is the graph with
$V(G[A])=A$ and $E(G[A])=\{ uv \in E(G): u,v \in A  \}$. Similarly, for a non-empty subset
$R$ of edges of $G$, the edge-induced subgraph $G(R)$ has $R$ as the set of edges and the set of endpoints of
$R$ as the vertices.
A path of length $n$ between two vertices $u$ and $v$ is a finite sequence of vertices $x_0=u, x_1, \cdots, x_n=v$ such that $x_i$ and $x_{i+1}$ are neighbors for $i=0,1,\cdots n-1$.
 A tree $T$ is a graph such that for any two vertices $u, v \in V(T)$, there exists a unique path joining
$u$ to $v$. This path (viewed as in induced subgraph of $T$) will be denoted by $[uv]$. We say that a graph $G$ contains arbitrarily long paths without branches if for any $n$, there exist vertices
$v_1, \dots, v_n \in V(G)$ such that $v_i$ is connected to $v_{i+1}$ for $1 \le i \le n$ and for $  1 \le i \le n$, the degree of $v_i$ in $G$ is exactly $2$.

\section{Amenability of Graphs}\label{amen}
In this section, we recall the definition of amenability for graphs and groups. Moreover, we will give a complete characterization of amenable trees without leaves.
\begin{definition}
Let $A$ be a subgraph of  graph $G$. The {\it boundary} of $A$ consists of those vertices of $A$ which are
connected to at least one vertex outside $A$. We will denoted the boundary by $\partial A$. Hence,
$$
\partial A=\{v\in V(A)\ :\ vu\in E(G),\ \mbox{for some } u\in V(A^c)\}\;,
$$
where $|A|=|V(A)|$.
\end{definition}

We will also need the following definition.

\begin{definition}
Define the {\it isoperimetric number} or {\it Cheeger constant} of a graph $G$ as follows
$$
i(G)=\inf \left\{ \frac{|\partial A|}{|A|}\ :\ A \mbox{ is a non-empty and finite subgraph of } G \right\}.
$$
\end{definition}
The graph $G$ is called {\it amenable} if $i(G)=0$.

Equivalently, the graph $G$ is amenable if and only if there is a sequence of finite subgraphs $(A_n)_{n\geq1}$ of $G$
such that
$$
\lim_{n\to\infty}\frac{|\partial A_n|}{|A_n|}=0.
$$
Such a sequence $(A_n)_{n\geq1}$ witnessing the amenability is called a F{\o}lner set.
Let us make two remarks about F{\o}lner set: first, one can always exchange a F{\o}lner set with one
which consists of finite connected graphs. This follows from the following easy lemma:
\begin{lemma}
If $A$ is a finite subgraph of $G$, such that $|\p A|\leq\epsilon |A|$, then there is connected subgraph $B$ of $A$, such that
$|\p B|\leq\epsilon |B|$.
\end{lemma}
\begin{proof}
Let $A_1,A_2,\cdots, A_n$ are finite connected components of $A$. Since $A_i$ are pairwise disjoint, we have $|\partial A|=|\partial A_1|+\cdot+|\partial A_n|$. Therefore,
$|\partial A_1|+\cdots+|\partial A_n|=|\partial A|\leq\epsilon |A|=\epsilon|A|_1+\cdots+ \epsilon|A_n|$. Hence, there exists $1 \le i \le n$ such that
$|\p A_i|\leq\epsilon |A_i|$
\end{proof}

Second, the F{\o}lner set can be chosen to exhaust $V(G)$:

\begin{proposition}
Let $(A_n)_{n\geq1}$ be a F{\o}lner sequence for the graph $G$.
Then there exists F{\o}lner sequence $(A'_n)_{n\geq1}$ such that $\bigcup_{n\geq0}A'_n= G$.
\end{proposition}
\begin{proof}
Let $B_n$ be a sequence of finite subgraphs such that their union is the whole  graph $G$.
By induction, define $A'_n=B_n\cup A_{k_n}$, where $\sqrt{|A_{k_n}|}\geq |B_n|$.  Then
$$
\frac{|\p A'_n|}{|A'_n|}\leq\frac{|B_n|+|\p A_{k_n}|}{|A_{k_n}|}\to0
$$
as $n$ goes to infinity.
\end{proof}
We can now give the definition of amenability for countable groups.
\begin{definition}
A countable group $G$ is called amenable whenever its Cayley graph admits a F{\o}lner set, i.e., there exists a sequence of finite subsets
$(A_n)_{n\geq1}$ such that
$$
\lim_n\frac{|gA_n\triangle A_n|}{|A_n|}=0
$$
for every $g\in G$, where $gA_n=\{ga\ :\ a\in A_n\}$ and $A_n\Delta gA_n$ is the symmetric difference of two sets $A_n$  and $gA_n$.
\end{definition}
Let $G$ be a finitely generated group with a symmetric finite set $\Sigma$.
The Cayley graph of the group $G$ is a graph whose vertices are elements  of $G$ and
two vertices of $g_1$ and $g_2$ are connected if there is an element $s\in \Sigma$ such that
$g_1s=g_2$. A finite generated group $G$ is thus amenable if and only if its Cayley graph is amenable.

\begin{proposition}\label{regular tree}
Let $T$ be a tree which does not have any vertex with degree less than 3.
Then for every finite subtree $A$, we have $|A|\leq 2|\p A|$. Consequently,
$T$ is not amenable.
\end{proposition}
\begin{proof}
Let $A$ be a finite subtree of $T$. Define $K_i=\{v\in A\ :\ \mbox{deg}v=i\ \}$. Hence, the boundary of $A$ at least includes the vertices with degree 1 and 2, hence $|\p A|\geq |K_1|+|K_2|$. It is well-know that
\begin{equation}\label{ed}
2|E(A)|= \sum_{v \in V(G)} deg(v) = \sum_{i \ge 1}i|K_i|
\end{equation}
On the other hand
\begin{equation}\label{ver}
|V(A)|= \sum_{i \ge 1} |K_i|.
\end{equation}
Since $A$ is a tree, $|E(G)|=|V(A)|-1$. Equalities \eqref{ed} and \eqref{ver} now imply
$$ |K_1|=\sum_{i\geq 3}(i-2)|K_i|+2\;.
$$
We now have $2|\partial A| \ge |K_1|+|K_2|+ \sum_{i\geq 3}(i-2)|K_i| \ge \sum_{j\geq 1}|K_j|$.
\end{proof}

As mentioned, Woess \cite[p.~114]{Woess2000} classified amenable trees without any leaves.
Here we provide an alternative proof for this result.

\begin{theorem}\cite[p.~114]{Woess2000}\label{thesis}
Let $T$ be an infinite tree with no leaves. Then $T$ is amenable if and only if $T$
contains  arbitrarily long paths without any branch.
\end{theorem}
\proof Let $T$ contain arbitrary long paths without branches. For each $n$, let
$A_n$ be a finite subtree  of $T$ with exactly $n$ vertices with degree 2.
Hence, $|A_n|=n+1$ and $|\p A_n|\leq2$. Consequently, $(A_n)$ is a F{\o}lner set.

Assume $T$ does not contain arbitrarily long paths without branches.
Let $T'$ be a tree obtained after removing all vertices of $T$ whose degrees are 2.
Since, $T$ does not have any leaves, by Proposition~\ref{regular tree}, $T'$ is not amenable.
If $A$ is a finite subtree of $T$, then
corresponding subtree $A'$ in $T'$  has the same boundary as $A'$ and clearly
$|A'|\leq |A|$. In addition, each edge of $A'$ is obtained by removing at most $d$ vertices,
where $d$ is the longest path without any branch.
In other words, $|A'|\leq d|A|$, and
\begin{equation}\label{sandewich}
\frac{|\p A'|}{|A'|}\leq\frac{|\p A|}{|A|}\leq d\frac{|\p A'|}{|A'|}.
\end{equation}
Combining the preceding inequalities and the fact $T'$ is not amenable imply
non-amenability of $T$.

The preceding theorem is not true if the trees are allowed to have infinitely many leaves, as the following example shows.

\begin{example}\label{amenable trees with leaves}
Let $T$ be a tree which can be obtained by attaching one vertex and one edge to the
Cayley graph of $\Bbb Z$ with respect to the generating set $\{-1,1\}$ (see, Figure~I).
Then $T$ is amenable, but $T$ does not contain arbitrarily long paths without any branch.

\end{example}
\begin{center}
\begin{figure}[ht]
\begin{tikzpicture} 
\tikzstyle{solid node}=[circle,draw,inner sep=1.5,fill=black]
\node[solid node]{}
child[grow=up]{node[solid node]{}}
child[grow=left]{node{}}
child[grow=right]{node[solid node]{}
  child[grow=up]{node[solid node]{}
  }
     child[grow=right]{node[solid node]{}}
     child[grow=right]{node[solid node]{}
     child[grow=up]{node[solid node]{}}
     child[grow=right]{node[solid node]{}
     child[grow=right]{node{}}
     child[grow=up]{node[solid node]{}
                }
  }
}
}
;
\end{tikzpicture}
\caption{Example~\ref{amenable trees with leaves}}
\end{figure}
\end{center}

\begin{remark}
One has to note that some properties of amenability for graphs may diverge from amenability of groups.
For instance, it is known that every subgroup of a (discrete) amenable group is amenable. The analogous property does
not hold for trees. This can be easily seen as follows: let $T$ be a $3$-regular tree and $T'$ be the tree obtained by
adding an infinite ray $Z$ (the graph with vertices $1,2, \dots$ where $i$ and $j$ are adjacent when $|i-j|=1$) to one of the vertices of $T$. In other words, let $T'$ be the graph obtained by taking the disjoint union of $T$ and $Z$ and identifying vertex $1$ of $Z$ with one of the vertices of $T$.
Clearly $T'$ contains arbitrary long paths without any branch and is hence
amenable, but it contains $T$ as a subtree which is not amenable. One can modify this example to construct an example
of a tree $T'$ and a subtree $T$ with the same set of ends as $T$ such that $T'$ is amenable, but $T$ is not.
\end{remark}

\section{Trimming and inessential Subtrees}\label{trim}
In this section, we will define certain operators on the space of trees. These definitions will later
be used to give a criterion for amenability of trees.

\begin{definition} Let $T$ be an infinite tree. The trimming operator $\Theta(T)$ is defined by
$$\Theta(T)= T[V(T) \setminus L(T)].$$
In other words, $\Theta(T)$ is the tree obtained by removing
all the leaves of $T$ together with their incident edges.

\begin{center}
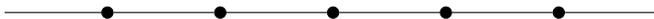
\begin{figure}[ht]
\begin{tikzpicture}
\tikzstyle{solid node}=[circle,draw,inner sep=1.5,fill=black]
\node[solid node]{}
child[grow=left]{node[solid node]{}
 child[grow=left]{node{}}}
child[grow=right]{node[solid node]{}
  child[grow=right]{node[solid node]{}
  child[grow=right]{node[solid node]{}
  child[grow=right]{node{}}
   }
  }
};
\end{tikzpicture}
\caption{Example~\ref{amenable trees with leaves} after trimming}
\end{figure}
\end{center}
\begin{remark}
We will always view $ \Theta(T)$ as an induced subtree of $T$. For instance,
$T$ has no leaves iff $\Theta(T)=T$. Note that $$T \supset \Theta(T) \supset \Theta^2(T)
\supset \cdots $$ is a decreasing sequence of subtrees of $T$. Also define, $\Theta^0(T)=T$ and $\Theta^{k+1}(T)= \Theta( \Theta^k(T))$, for $k \ge 0$.
\end{remark}

We say that the trimming stops in finite time if this sequence stabilizes at some point,
i.e., if there exists $k \ge 0$ such that $\Theta^k(T)$ does not have any leaves, which is
equivalent to $\Theta^{l}(T)= \Theta^{k}(T)$ for all $l \ge k$.
Otherwise, we say that $T$ can be trimmed indefinitely.

\begin{example}\label{trimmed indefinitely}
Although $\Theta$ is defined on the space of trees, it also induces a map on the space $\mathcal{T}$ of isomorphism classes of trees. Let us denote the isomorphism class of a tree $T$ by $[T]$.
We can now study the
dynamics of $\Theta$ on $\mathcal{T}$ and pose various questions about it. For instance, finite trees $[T]$ are exactly those trees whose orbit contains the empty tree (the tree with one vertex and no edge). For a more
interesting example, let $T$ be the tree with $V(T)=\{ (i,j) \in \ZZ^{2}: 0 \le j \le i \}$, and the edges
between $(i,0), (i+1,0)$ for $i \ge 0$ and $(i,j)$ and $(i,j+1)$ for every $i \ge 1$ and $0\le j \le i-1$, see  Figure~3.
It is easy to see that $T$ can be trimmed indefinitely, while $\Theta(T)$ is isomorphic to $T$. In other words, $[T]$
is a fixed point for $\Theta$. Clearly, if $L(T) =\emptyset$, then $\Theta(T)=T$. It is an
interesting question to characterize those trees with $L(T) \neq \emptyset$, for which $\Theta([T])=[T].$
One can analogously define for any $n \ge 1$ the tree $T_n$ by
\[  V(T_n)=\{ (i,j) \in \ZZ^{2}: 0 \le j \le ni \}, \]
with the edges similar to those of $T$ above. In this case, one can see that $[T_n]$ is a periodic point of $\Theta$ with
the smallest period $n$.
\end{example}
\begin{center}
\begin{figure}[ht]\label{graph trimmed indefinitely}
\begin{tikzpicture}
\tikzstyle{solid node}=[circle,draw,inner sep=1.5,fill=black]
\node[solid node]{}
  child[grow=right]{node[solid node]{}
       child[grow=up]{node[solid node]{}}
  child[grow=right]{node[solid node]{}
       child[grow=up]{node[solid node]{}
         child[grow=up]{node[solid node]{}}
         }
         child[grow=right]{node[solid node]{}
             child[grow=right]{node[solid node]{}
               child[grow=right]{node{}}
               child[grow=up]{node[solid node]{}
               child[grow=up]{node[solid node]{}
               child[grow=up]{node[solid node]{}
               child[grow=up]{node[solid node]{}}
               }
               }
               }
             }
          child[grow=up]{node[solid node]{}
            child[grow=up]{node[solid node]{}
            child[grow=up]{node[solid node]{}}
            }
            }
            }
            }
    }
  ;
\end{tikzpicture}
 \caption{Example~\ref{trimmed indefinitely}}
 \end{figure}
\end{center}
We will now define the notion of inessential subtree and use it to prove some of properties
of finitely trimable trees.

\begin{definition}
Let $T_{0}$ be a finite (connected) subtree of an infinite tree $T$ with at least one edge. We say that $T_{0}$ is inessential if the edge-induced subgraph $G(E(T) \setminus E(T_{0}))$ is connected, and hence is a tree.
\end{definition}

For an inessential subtree $T_{0}$ of $T$, set $\overline{T_{0}}=G(E(T) \setminus E(T_{0}))$, and note that $E(G)$ is the disjoint union of the edge set
of the trees $T_{0}$ and $\overline{T_{0}}$. This implies that $T_{0}$ and $\overline{T_{0}}$ have exactly one vertex in common.
We call this vertex the root of $T_{0}$ and denote it by $\rac(T_{0})$. Note that since $\rac(T_{0})$ is adjacent to a vertex in $T_{0}$, as well as one in $\overline{T_{0}}$, its degree is at least $2$. Let us derive some immediate consequences of this definition.

\begin{proposition}
If $T_{1}$ and $T_{2}$ are inessential subtrees of an infinite tree $T$ with $\rac(T_{1})=\rac(T_{2})$,
then $T_{1} \cup T_{2}$ is also an inessential subtree of $T$.
\end{proposition}

\begin{proof}
Since $T_{1}$ and $T_{2}$ have a common vertex, $T_{1} \cup T_{2}$ is a connected, and hence a subtree of $T$.
Also $ \overline{T_{1} \cup T_{2}}=   \overline{T_{1}} \cap \overline{T_{2}} $ is the intersection of two
trees that have a common vertex $\rac(T_{1})=\rac(T_{2})$, and is hence connected.
\end{proof}

\begin{proposition}\label{crit}
An infinite tree has an inessential subtree if and only if it has a leaf.

\end{proposition}
\begin{proof}
Let $v$ be a leaf of $T$ and $w$ be the unique vertex of $T$ adjacent to $v$. Then the single edge $vw$
is an inessential subtree of $T$. Conversely, if $T'$ is an inessential subtree of a tree $T$,
then $T'$ has at least two vertices of degree $1$. Call them $u'$ and $v'$. We claim that at least one of
$u'$ and $v'$ has degree $1$ in $T$. If, on the contrary, $u'$ and $v'$ are adjacent to vertices
$u$ and $v$ in $V(T) \setminus V(T')$, respectively, then the unique path from $u$ to $v$ will
contain both $u', v' \in V(T')$. Now consider the path joining $u$ and $v$ in $G(E(T) \setminus E(T'))$. The
unique path between $u'$ and $v'$ in $T$ must contain this path, and will hence depart $T'$. This contradicts
the assumption that $T'$ is connected.
%
%
\end{proof}

\begin{lemma}\label{add}
If $\Theta(T)$ contains an inessential subtree with $k$ vertices, then $T$ contains an inessential tree with at least
$k+1$ vertices.
\end{lemma}

\begin{proof}
Let $T_0$ be an inessential subtree of $\Theta(T)$ and $v$ be one of its leaves. Since $v$ is not a leaf of $T$, it must
be connected to at least one vertex of $V(T) \setminus V(\Theta(T))$, which is automatically a leaf of $T$. The tree added by adding all such vertices $w$ and the corresponding edges $vw$ to $T_0$ is connected and hence an inessential subtree of $T$ with at least
$k+1$ vertices.
\end{proof}

\begin{proposition}\label{indef}
Let $T$ be an infinite tree that can be trimmed indefinitely. Then $T$ is amenable.
\end{proposition}

\begin{proof}
We show that $T$ contains an inessential subtree with arbitrarily
large number of vertices. Since $\Theta^k(T)$ contains a leaf, hence it contains an inessential subtree by Proposition~\ref{crit}. Now, a repeated application of Lemma~\ref{add} shows that $T$ contains an inessential tree $T_0$ with at least $k$ vertices. Let $S$ be the set of all vertices of $T_0$ except for the root. It is clear that $|S|= k-1$ and
$ \partial S= \{ \rac(T_{0}) \}$. This shows that
\[ \frac{|\partial S|}{|S|}= \frac{1}{k-1}. \]
Hence, by letting $k \to \infty$, we obtain a sequence of F{\o}lner sets in the tree.
\end{proof}

\begin{lemma}\label{reduce}
Let $T$ be an infinite tree such that $\Theta(T)$ is amenable. Then $T$ is amenable.
\end{lemma}

\begin{proof}
Let $S$ be a connected subgraph of $\Theta(T)$ with {$\dfrac{|\p S|}{|S|} < \epsilon$}. Set
$ \overline{S}$ to be the connected subgraph of $T$ obtained by adding all of the leaves of $T$ that are connected to a vertex of $S$. Note that since the only vertices of $T \setminus \Theta(T)$ are leaves of $T$, the boundary of $ \overline{S}$ in $T$ is equal to the boundary of $S$ in $ \Theta(T)$. Since
$| \overline{S}| \ge |S|$, we have $ \dfrac{|\p\overline{S}|}{|\overline S|}\leq\epsilon$.
\end{proof}

\end{definition}

We can now prove the main result of this section which is the generalization of Theorem~\ref{thesis}

\begin{theorem}
Let $T$ be an infinite tree. Then $T$ is amenable if and only if $T$ contains arbitrarily large inessential trees or for some $k \ge 1$, $\Theta^k(T)$ has arbitrarily long paths without branch. Moreover, the former is always the case if $T$ can be trimmed indefinitely.
\end{theorem}

\begin{proof}
The same argument as in Proposition~\ref{indef} shows that if $T$ contains arbitrarily large inessential trees, then it is amenable. Hence,  without loss of generality, we can assume that there exists $k \ge 1$ such that $\Theta^k(T)$ does not have any leaves but contains arbitrarily long paths. The proof now follows by induction on $k$. For $k=0$, we can take these long paths without branch as F{\o}lner
sets. Since $\Theta^k(T)=\Theta^{k-1}(\Theta(T))$, by induction hypothesis, we obtain that $\Theta(T)$ is amenable. Hence, using Lemma \ref{reduce}, $T$ is amenable.

Let us now prove the converse.
Assume that exists $k \ge 0$ such that $\Theta^{k}(T)$ contains no leaves and does
 not contain arbitrarily long paths, and that the
largest inessential tree of $T$ has cardinality $R$. We will show that $T$ is  nonamenable.

We will assume that $\Theta^{k}(T)= \Theta^{k+1}(T)=T'$, i.e.,
$T$ can only be trimmed $k$ times and $T'$ does not have
arbitrarily long paths.
This implies that $T$ is obtained from $T'$ by adding a (possibly infinite) number of
inessential trees, each attached at a distinct vertex of $T$.

Let us now assume that $A$ is a connected subgraph of $T$.

Let $A'=A\cap V(T')$. We claim that
\[ |A'| \ge \frac{1}{R} |A| \]
Also the cardinality of the boundary of $A'$ in $T'$ is the same
as the cardinality of the boundary of $A$ in $T$. This implies that
 \[ \frac{ | \partial A'|   }{|A'|} \le R \frac{| \partial  A|}{|A|} . \]

We have now reduced the problem to the case that the tree does not have
any vertices of degree $1$.
\end{proof}

\section{Application: Amenability of Random Trees}\label{gw}
In this section, we will consider the question of amenability for Galton-Watson trees. First we will give some basic
definitions regarding the Galton-Watson process. For details, the reader is referred to \cite{AN}.

Let $\pi=(p_n)_{n \ge 0}$ be a distribution on the set of non-negative integers. We will define the Galton-Watson
process associated to $\pi$ as a distribution on the set of rooted labeled trees. We will start by setting up the notations that will be used in this section.

A Galton-Watson tree is always designated with a distinguished vertex refereed to as the
root and denoted by $\rt$. Set
\[ {\mathcal I}= \{ \rt \} \cup \bigcup_{j=1}^{ \infty} \NN^j \]
where $\NN$ is the set of positive integers. For each $I= (i_1, \dots , i_k) \in {\mathcal I}$, we define
its length (or generation) by $|I|=k$. We will also set $| \rt |=0$. A set ${\mathcal J} \subseteq {\mathcal I}$ is called inductive if it satisfies the following properties:
\begin{enumerate}
\item  $ \rt \in {\mathcal J}$
\item For each $k \ge 1$ and $I=(i_1, \dots, i_k) \in {\mathcal J} $, we have
$ \hat{I}:=(i_1, \dots , i_{k-1}) \in {\mathcal J}$.
\item For each $k \ge 1$ and $I=(i_1, \dots, i_k) \in {\mathcal J} $, if $i_k \ge 2$,
then $(i_1, \dots , i_{k-1}, i_k-1) \in {\mathcal J}$
\end{enumerate}

If $\hat{J}=I$, we say that $I$ is an ancestor of $J$, and that $J$ is a descendant of $I$.
Intuitively, an inductive set is a set which is closed with respect to the ancestor operation and moreover, the set of offsprings
of any vertex are always labelled from $1$ to $k$ for some $k \ge 0$.
Let $X_I$, $I \in I$ be a sequence of independent identically distributed random variables with distribution $\pi$. A random rooted tree is constructed as follows: the root
$v_\rt$ has $X_{ \rt}$ direct offsprings (also called children) denoted by $v_{(i_1)}$, for $1 \le i_1 \le X_{\rt}. $ These vertices are called the first generation. From here, the construction continues inductively. Assume that the vertices of generation $\ell$ have been constructed. Each vertex in generation $\ell$ is of the form $v_I$ for some $I=(i_1, \dots, i_{\ell})$, hence $|I|= \ell$. The vertext $v_I$ has $X_I$ children, namely $v_{(i_1, \dots, i_{\ell}, i_{\ell +1})}$, where $i_{ \ell+1}$ ranges from $1$ to $X_I$. It is easy to see that the set of
vertices of the tree ${\mathcal T}$ thus constructed is of the form $V({\mathcal T})=\{v_J: J \in {\mathcal J} \}$, where $ {\mathcal J}$ is an inductive set. We will denote the set of vertices in generation $\ell$ by $ V({\mathcal T})_{\ell}$ and set
$W_{\ell}=| V( {\mathcal T})_{\ell}|$. The rooted subtree of $ {\mathcal T}$ consisting of all the offsprings of vertex $v_I$ rooted at $v_I$ will be denoted by $ {\mathcal T}^I$. The (finite) rooted subtree of $ {\mathcal T}$ consisting of all vertices of the first $k$ generations will be denoted by $ {\mathcal T}_k$.

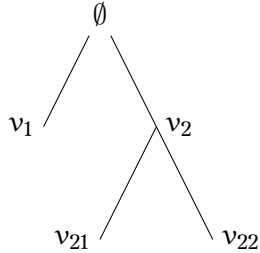
\begin{center}
\begin{tikzpicture}
  \node (root) {$\rt$}
child child {
      child {coordinate (special)}
child };
  \node[left] at (root-1) {$v_1$};
  \node[right] at (root-2) {$v_2$};
  \node[left] at (special) {$v_{21}$};
  \node[right] at (root-2-2) {$v_{22}$};
\end{tikzpicture}
\end{center}

\begin{remark}
Before we proceed to the proof, a few remarks are in order. There is a vast literature
on the Galton-Watson process. Let $m= \ex{X_I}= \sum_{j=0}^{ \infty} j p_j$ be the expected number of children of any vertex. It is a classical theorem that if $p_1 \neq 1$, then $ {\mathcal T}$ is almost surely an infinite tree iff $m>1$.
If $p_1=1$ then $T$ will be isomorphic to an infinite path and hence amenable. From now on we will always exclude this case. The cases $m=1$ and $m<1$ are usually referred to as the critical and subcritical case. Since every finite tree is by definition amenable, we can condition on the non-extinction of $ {\mathcal T}$.
\end{remark}

\begin{proof}[Proof of Theorem \ref{galton}]
First note that if $p_0=p_1=0$, then $ {\mathcal T}$ is almost surely infinite and the degree of every vertex, except possibly for the root, is at least $3$. Such a tree is clearly non-amenable. We will now show that if $p_0>0$ or $p_1>0$, then the tree is amenable.
Let $m= \sum_{k=1}^{ \infty } k p_k$.
By the above remark, we can assume that $m>1$.
First assume that $m< \infty$. We know that
$\ex{W_n}=m^n$. We will distinguish two cases:

{\it Case 1:} Assume that $p_0=0$ and $0<p_1<1$. We will start by two observations.
First, it is easy to see that in this case
$W_n$ is a non-decreasing sequence. Moreover, $W_{n+1}=W_n$ if each vertex in generation $n$ has exactly one offspring. This
implies that if $k \ge 1$,
$$\pr{W_{n+1}>W_n|W_n=k}= 1-p_1^k \ge 1-p_1>0.$$
Hence $\pr{W_{n+1}>W_n} \ge 1- p_1$, and an application of Borel-Cantelli shows that $\pr{W_n \to \infty}=1$. Second, with probability $q=p_1^{d+1}>0$
each vertex in the the first $d$ generations has exactly one child, i.e.,
the tree $ {\mathcal T}_{d+1}$ is isomorphic to $P_{d+1}$.

For any vertex $v_I$ in the $n^{\text{th}}$ generation, consider the rooted subtree $ {\mathcal T}^I$ at
$v_I$. Note that of ${\mathcal T}^I$ are i.i.d. random trees with the same distribution as $ {\mathcal T}$. By the second observation above,
each ${\mathcal T}^I$ has probability $q$ of being isomorphic to a path of length $d+1$.
Let $A_n$ be the event that at least one of these subtrees is isomorphic to a path of
length $d+1$. For a fixed $r$, we have
\[ \pr{A_n} \ge \pr{A_n|W_n>r}\pr{W_n>r}= (1-(1-q)^r)\pr{W_n>r}\to 1-(1-q)^r, \] as
$n \to \infty$. Since $r$ is arbitrary, we have $\pr{A_n} \to 1$, as $n \to \infty$. This means that with probability $1$, $ {\mathcal T}$ contains a path of length $d+1$ for every $d \ge 1$, which proves the almost sure amenability of $ {\mathcal T}$.

{\it Case 2:}
Let us now consider the case that $p_0>0$. Note that in this case there is no guarantee that $W_n \to  \infty$ as $n \to \infty$.
Fix $ d \ge 1$. We will show that, with probability $1$, the isoperimetric constant of $ {\mathcal T}$ is at most $1/d$. The large F{\o}lner sets in this case arise from the following dichotomy: for an appropriate value of $n \gg 1$: (a) either there are ``many'' vertices in generation $n$, in which case, with high probability, the subtree of ${\mathcal T}$ starting from one of them must terminate exactly after $d$ generations, i.e., the first $d$ generations starting from one of these vertices must be a finite tree with all vertices of degree $1$ in generation $d$, or, (b) there are ``few'' vertices in generation $n$, which implies that the the first $n-1$ generation of the graph forms a large set with a small boundary. Let us make this idea precise.
Fix $d \ge 1$, choose $s>0$ such that $p_s>0$. Let $T_{s,d}$ denote the (deterministic) finite rooted $s$-ary tree of depth $d$, that is, a rooted tree, where starting from root up to generation $d-1$, each vertex has exactly $s$ children, but the vertices in generation $d$ have not any children.  Let $q$ be the probability that $ {\mathcal T}_{d+1}$ is isomorphic $T_{s,d}$. Since $p_0>0, p_s>0$, we have $q>0$.

Let $E$ denote event that $ {\mathcal T}$ is finite, or equivalently, that $W_n=0$ for $n \gg 1$.
It suffices to show that $ {\mathcal T}$ is amenable conditioned on $E^c$. Let $A_d$ denote the event that $ {\mathcal T}$ contains a subtree with isoperimetric constant at most $1/d$. For any $r \ge 1$, choose $n \ge rd$, and note that given that ${\mathcal T}$ is infinite and $W_n>r$, $A_d$ will take place if at least one of the subtrees starting from one of the vertices in generation $d$ is isomorphic to $T_{s,d}$. Since there are at least $r$ vertices in generation $n$, we have
\[ \pr{A_d|E^c \cap \{ W_n>r \} } \ge 1-(1-q)^r. \]

On the other hand, if there are at most $r$ vertices in generation $n$, then the subtree ${\mathcal T}_{n-1}$, which has at least $n$ vertices, has a boundary of size at most $r$, implying that its isoperimetric constant is at most $r/n \le 1/d$. Hence,
\[ \pr{ A_d|E^c \cap \{W_n \le r \} }=1. \]
Combining the two case, we have that for given $r, d\ge 1$,
 \[ \pr{ A_d|E^c  }\ge 1-(1-q)^r.  \]
Since $r$ is arbitrary and $q>0$, we deduce that for any $d \ge 1$, $\pr{ A_d|E^c  }=1$, implying that $ {\mathcal T}$ almost surely has a set with
isoperimetric constant at most $1/d$, for every $d$, proving the almost sure amenability
of $ {\mathcal T}$ in this case.
\end{proof}
\bibliographystyle{alpha}
\bibliography{Ref}

\begin{thebibliography}{AAV13}

\bibitem[AAV13]{AAV}
Gideon Amir, Omer Angel, and B{\'a}lint Vir{\'a}g.
\newblock Amenability of linear-activity automaton groups.
\newblock {\em J. Eur. Math. Soc. (JEMS)}, 15(3):705--730, 2013.

\bibitem[AN72]{AN}
Krishna~B. Athreya and Peter~E. Ney.
\newblock {\em Branching processes}.
\newblock Springer-Verlag, New York-Heidelberg, 1972.
\newblock Die Grundlehren der mathematischen Wissenschaften, Band 196.

\bibitem[BKN10]{BKN}
Laurent Bartholdi, Vadim~A. Kaimanovich, and Volodymyr~V. Nekrashevych.
\newblock On amenability of automata groups.
\newblock {\em Duke Math. J.}, 154(3):575--598, 2010.

\bibitem[F{\o}l55]{Folner55}
Erling F{\o}lner.
\newblock On groups with full {B}anach mean value.
\newblock {\em Math. Scand.}, 3:243--254, 1955.

\bibitem[Ger86]{Gerl86}
Peter Gerl.
\newblock Eine isoperimetrische {E}igenschaft von {B}\"aumen.
\newblock {\em \"Osterreich. Akad. Wiss. Math.-Natur. Kl. Sitzungsber. II},
  195(1-3):49--52, 1986.

\bibitem[Ger88]{Gerl87}
Peter Gerl.
\newblock Amenable groups and amenable graphs.
\newblock In {\em Harmonic analysis ({L}uxembourg, 1987)}, volume 1359 of {\em
  Lecture Notes in Math.}, pages 181--190. Springer, Berlin, 1988.

\bibitem[Kai05]{Kaimanovich2005}
Vadim~A. Kaimanovich.
\newblock ``{M}\"unchhausen trick'' and amenability of self-similar groups.
\newblock {\em Internat. J. Algebra Comput.}, 15(5-6):907--937, 2005.

\bibitem[Kes59]{Kesten59}
Harry Kesten.
\newblock Full {B}anach mean values on countable groups.
\newblock {\em Math. Scand.}, 7:146--156, 1959.

\bibitem[vN29]{Neumann}
John von Neumann.
\newblock Zur allgemeinen theorie des ma{\ss}es.
\newblock {\em Fundamenta Mathematica}, 13:73--116, 1929.

\bibitem[Woe00]{Woess2000}
Wolfgang Woess.
\newblock {\em Random walks on infinite graphs and groups}, volume 138 of {\em
  Cambridge Tracts in Mathematics}.
\newblock Cambridge University Press, Cambridge, 2000.

\bibitem[Zim78]{Zimmer78}
Robert~J. Zimmer.
\newblock Amenable ergodic group actions and an application to {P}oisson
  boundaries of random walks.
\newblock {\em J. Functional Analysis}, 27(3):350--372, 1978.

\end{thebibliography}
\end{document}